\documentclass[a4paper,oneside]{amsart}

\usepackage{packages}

\title{The orbit space approach for piecewise smooth vector fields}

\author[O. M. L. Gomide]{Ot\'avio M. L. Gomide}
\address[OMLG]{Department of Mathematics, UFG, IME\\ Goi\^ania-GO, 74690-900, Brazil.}
\email{otavio.marcal@ufg.br}

\author[P. Mattos]{Pedro G. Mattos}
\address[PM]{Department of Mathematics, Unicamp, IMECC\\ Campinas-SP, 13083-970, Brazil}
\email{pedrogmattos@ime.unicamp.br}

\author[R. Var\~ao]{R\'egis Var\~ao}
\address[RV]{Department of Mathematics, Unicamp, IMECC\\ Campinas-SP, 13083-970, Brazil}
\email{varao@unicamp.br}



\begin{document}

\begin{abstract}
In this work we develop a well-defined theory of \textit{orbit spaces} for piecewise smooth vector fields (PSVFs). This approach is inspired by the techniques already used in the study of endomorphisms, namely inverse limit analysis, and has been used before for PSVFs \cite{ACV, EMV}. We then apply the construction of our theory to understanding transitivity in PSVFs. Our results prove that the known examples of transitive PSVFs in the literature, the \textit{bean model} \cite{BCE} and the \textit{sphere model} \cite{EJV}, are indeed transitive in the orbit space.
\end{abstract}

\maketitle


\section{Introduction}
\label{sec:introduction}

Piecewise smooth vector fields (PSVFs) have been extensively studied in the last years due to their applicability to model real world problems \cite{liv:Jeffrey-ModelingNonsmoothDynamics}. In light of this, there is a natural interest in understanding the dynamics associated with them, which is very complicated and presents fascinating behaviors \cite{EJV}.

The novelties coming from PSVFs compared to the continuous case resides on a proper or convenient way to define a solution and its behavior when two distinct vector fields meet, namely at the \textit{switching} (or \textit{discontinuity}) \textit{manifold}. There are several ways to define the solutions of a PSVF, and each of them gives rise to a new dynamics \cite{utkin01, utkin03}; nevertheless, there is a special interest in the dynamics given by \textit{Filippov's convention} \cite{filippov-book}, since it provides very accurate models for different kinds of problems \cite{liv:Jeffrey-ModelingNonsmoothDynamics}.
Filippov's convention can be interpreted broadly \cite{liv:Jeffrey-HiddenDynamics}, but under the most common conception it establishes that the valid solutions to a PSVF --- called \textit{Filippov orbits} --- are concatenations of regular orbits going through the regular regions of the space with \textit{sliding orbits} over the switching manifold, more precisely in the \textit{sliding} and \textit{escaping regions}. On these regions, there are many possibilities for the concatenation, so consequently PSVFs presents nonuniqueness of solutions, an important
feature of the  rich and complicated structure of their dynamics.
Many orbits of a PSVF can visit the same point, which is a behavior that does not occur for continuous vector fields, and this prevents the existence of a well-defined flow associated to the PSVF.

Many works have studied the dynamics of PSVFs (see \cite{BONET2017142,GST-generic_bifurcations,NOVAES20154615} and references therein),
but our work has as one of its goals to propose a new way to understand them. Because PSVFs under Filippov's convention allow multiple orbits to pass through the same point, this inherent nonuniqueness of solutions can obscure the true understanding of the system's overall behavior. Therefore in this work we provide a way to understand how the nonuniqueness of solutions impacts the dynamics of a PSVF, and we do so by introducing an associated dynamical system --- called the \textit{orbit space} of the PSVF --- which is able to restore the uniqueness property and also the existence of a well-defined flow. This expands upon previous work \cite{ACV}.

The orbit space associated to a PSVF is the space of all maximal Filippov orbits of the PSVF.
This is inspired by the technique known as \textit{inverse limit}, widely studied in the context of endomorphisms \cite{varao-cantarino, przytycki1976anosov, endoBOOK}. An endomorphism is a noninvertible map, and the inverse limit is simply the space of all possible (discrete) orbits. There exists a very natural dynamics in this space which is just a translation of the orbits of the endomorphism. We apply this idea to the PSVF, since it is also a noninvertible system (see \cref{sec:orbit_space}). The orbit space then carries all the complexity of the system, but without the issues of having multiple orbits going through the same point. Importantly, in the degenerate case the PSVF is a smooth system, the orbit space system is going to be isomorphic to the original smooth system, so its study in the general case is really an extension of the classical theory of smooth flows.

We aim to establish a well-defined theory of orbit spaces for PSVFs, expanding on previous works that have used the concept of orbit spaces (cf. \cite{ACV, EMV}). In this work, our focus will be on understanding the relation between the concept of topological transitivity on the PSVF and on the more appropriate orbit space. Among all properties of smooth dynamical systems, transitivity has always been one of the most fundamental ones. A transitive system is a system that has a dense orbit. This is an interesting concept by itself but it is worth mentioning that transitivity is also one of the ingredients of chaos \cite{Banks01041992, BARNSLEY, devaney, Silverman}.

In the literature, there are two models of chaotic $2$-dimensional PSVFs that have been shown to display a dense orbit in the phase space: the \textit{bean model} \cite{BCE} and the \textit{sphere model} \cite{EJV} (see \cref{fig:bean_and_sphere_models}).
These two articles explore intensively the nonuniqueness of solution, in particular the sliding and escaping regions. These are regions on which many different orbits go through the same points repeatedly. Escaping regions produce so many different orbits that one may have a feeling that transitivity is forcibly more common in the PSVF scenario. This is in some sense true, but one should see it as due to the richness of the dynamics. For instance, in \cite{EJV} the authors provide a first example of transitivity for a PSVF on a $2$-sphere, which is interesting and shows the richness of PSVFs, since a continuous transitive flow on the $2$-sphere is long known not to exist.
Despite the importance of PSVFs, there is still a lack of results about its global dynamics. We may claim that this is due to the scenario of nonuniqueness of solutions, which we propose to properly overcome in the present work. 

\begin{figure}
	\centering
	
	\begin{subfigure}[c]{0.5\textwidth}
		\centering
		\includegraphics{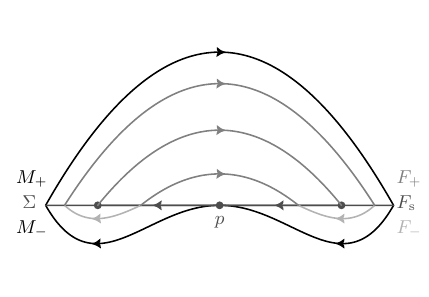}
		\caption[The bean model]{
			The bean model \cite{BCE}.
		}
		\label{fig:bean_model}
	\end{subfigure}
	\begin{subfigure}[c]{0.5\textwidth}
		\centering
		\includegraphics{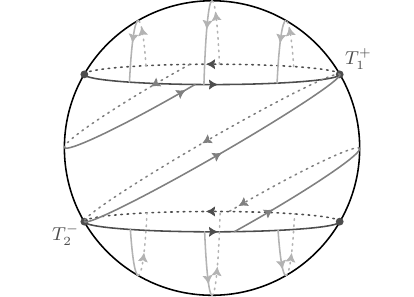}
		\caption[The sphere model]{
			The sphere model \cite{EJV}.
		}
		\label{fig:sphere_model}
	\end{subfigure}
	
	\caption[The bean and sphere models]{
		The bean and sphere models, PSVFs that are transitive. The highlighted points are the tangency points. They play a fundamental role in the topological transitivity of these systems by connecting stable and unstable sliding regions.
	}
	\label{fig:bean_and_sphere_models}
\end{figure}

\subsection{Main results}

Our first main result is showing that the orbit space we present is well defined. This has been done before \cite{ACV} but here we prove it for the more general cases in which the manifold is not necessarily compact and the PSVF is bounded (we also mention the proof for a Lipschitz continuous vector field). For a Manifold $M$ with metric $d$ and a Filippov field $F$ over $M$, we denote its orbit space by $\wt M$, its metric by $\wt d$ and the dynamics on $\wt M$ by $\wt \Phi$.

\begin{introtheorem}
\label{theo:orbit_space_is_metric_and_flow_is_continuous}
Let $M$ be a manifold and $F$ be a bounded Filippov system on $M$. Then $(\wt M, \wt d)$ is a metric space and $\wt \Phi$ is a continuous flow on $\wt M$.
\end{introtheorem}

Then we prove some topological properties of the metric space $\wt M$ that are useful for the study of topological transitivity. We call a topological space that has no isolated points a \textit{perfect} space.

\begin{introtheorem}
\label{theo:topological_properties_of_orbit_space}
Let $M$ be a manifold and $F$ be a bounded Filippov system on $M$ such that $\Tange$ is finite and each tangency point has finite multiplicity. Then the orbit space $\wt M$ is a separable, Baire and perfect metric space.
\end{introtheorem}

In particular this implies that topological transitivity in $\wt M$ is equivalent to the existence of a dense orbit, as a consequence of Birkhoff's transitivity theorem (\cref{theo:birkhoff.transitivity}). Finally, we relate topological transitivity in the PSVF system $(M, F)$ with topological transitivity in the orbit space system $(\wt M, \wt \Phi)$.

\begin{introtheorem}
\label{theo:transitivity_and_topological_transitivity_in_orbit_space}
Let $M$ be a manifold and $F$ be a bounded Filippov system on $M$ such that $\Tange$ is finite and each tangency point has finite multiplicity.
\begin{enumerate}
    \item If the orbit space flow $\wt \Phi$ is transitive, then the Filippov vector field $F$ is also transitive. \label{theo:transitivity_and_topological_transitivity_in_orbit_space(1)}
    \item If $F$ satisfies the specific assumptions presented in \cref{prop:main_transitivity} (which in particular imply transitivity of $F$), then $\wt \Phi$ is also topologically transitive. \label{theo:transitivity_and_topological_transitivity_in_orbit_space(2)}
\end{enumerate}
\end{introtheorem}

This theorem basically states that (\ref{theo:transitivity_and_topological_transitivity_in_orbit_space(1)}) transitivity  ``above'', i.e. in the orbit space system $(\wt M, \wt \Phi)$, ff fi implies transitivity ``below'', in the Filippov system $(M, F)$, and that (\ref{theo:transitivity_and_topological_transitivity_in_orbit_space(2)}) the converse is true for a specific class of Filippov systems. These specific assumptions have to do with how some of the tangency points of the Filippov system are connected to each other. The class of systems for which the converse is true includes the bean model \cite{BCE} and the sphere model \cite{EJV}, the two know examples of transitive Filippov vector fields in the literature. 

\subsection{Structure of the paper}

In \cref{sec:preliminaries} we present our notation and definitions related to piecewise smooth vector fields $F$ on a manifold $M$, and Filippov's convention which defines the piecewise smooth orbits that are accepted as solutions of $F$. This includes defining the escaping, sliding, crossing and tangent regions. We also define topological transitivity for $F$. In \cref{sec:orbit_space} we define the orbit space $\wt M$ of the system $(M, F)$, a flow $\wt \Phi$ on $\wt M$, and a distance function $\wt d$ on $\wt M$ which is related to a Riemannian distance $d$ on $M$. We prove \cref{theo:orbit_space_is_metric_and_flow_is_continuous} in \cref{ssec:properties_orbit_distance}, \cref{theo:topological_properties_of_orbit_space} in \cref{ssec:topology_orbit_space}, and \cref{theo:transitivity_and_topological_transitivity_in_orbit_space}
in \cref{ssec:transitivity}.

\section{Preliminaries}
\label{sec:preliminaries}

Let $M$ be a manifold.
Assume that there exists a compact embedded codimension $1$ submanifold $\Switch = h\inv(0)$ of $M$, where  $h\colon M \to \R$ is a $\Cont^r$  function ($r>1$ large enough) with $0$ as a regular value, which splits $M$ into the disjoint \emph{regular regions} $M_+ := \set{p \in M \st h(p) > 0}$ and $M_- := \set{p \in M \st h(p) < 0}$. We call $\Switch$ the \defemph{discontinuity manifold}, or \defemph{switching manifold}, (generated by $h$) of $M$.

Let $\mathfrak{X}^r$ denote the set of smooth vector fields of class $\Cont^r$. A \defemph{piecewise smooth vector field} on $M$ with discontinuity manifold $\Switch$ is a vector field of the form
    \begin{equation}
    \label{eq:definition.F}
    F(p)=\frac{F_+(p)+F_-(p)}{2}+\sgn(h(p))\frac{F_+(p)-F_-(p)}{2},
    \end{equation}
where $F_+,F_- \in \mathfrak{X}^r(M)$ and the sign function is multivalued at the origin and may assume values in $[-1,1]$. We denote $F = (F_+, F_-)$. Notice that $F$ equals the vector field $F_+$ on region $M_+$, the vector field $F_-$ on region $M_-$, and $F$ is multivalued on $\Switch$. The same construction can be made for more than $2$ regular regions, but we prefer to keep the presentation simple.

The \defemph{Lie derivative} of $h$ in the direction of a vector field $X \in \mathfrak{X}^r(M)$ at $p \in M$ is given by $Xh(p) = \left\langle \nabla h(p), X(p) \right\rangle_p$. In this work we assume that the solutions of a piecewise smooth vector field $F$ are given by \emph{Filippov's convention} \cite{filippov-book}. In this case, it is convenient to classify the points of $p \in \Switch$ into the following types:
\begin{enumerate}
	\item \defemph{tangency points}, which belong to the \defemph{tangent region}
		\begin{equation*}
		\Tange := \set{p \in \Slide \st F_+h(p)F_-h(p) = 0};
		\end{equation*}
	\item \defemph{crossing points}, which belong to the \defemph{crossing region}
		\begin{equation*}
		\Cross := \set{p \in \Slide \st F_+h(p)F_-h(p)> 0};
		\end{equation*}
	\item and \defemph{sliding points}, which belong to the \defemph{sliding region}
		\begin{equation*}
		\Slide := \set{p \in \Slide \st F_+h(p)F_-h(p)< 0}.
		\end{equation*}
	These are further subdivided into
		\begin{enumerate}
		\item the \defemph{stable sliding region} $\Slide[s] := \set{p \in \Slide \st F_+h(p)<0}$ and
		\item and the \defemph{unstable sliding region} $\Slide[u] := \set{p \in \Slide \st F_+h(p)>0}$ (also called the \defemph{escaping region}).
		\end{enumerate}  
\end{enumerate}

Notice that the discontinuity manifold $\Switch$ is the disjoint union of the sets $\Tange$, $\Cross$, $\Slide[s]$ and $\Slide[u]$. Also, the sets $\Cross$, $\Slide[s]$ and $\Slide[u]$ are relative open in $\Switch$.

In order to classify the tangency points we need to define higher order Lie derivatives of $h$. For $X_{1},\cdots, X_{k} \in \mathfrak{X}^r(M)$, the higher order Lie derivatives of $h$ are recursively defined by
    \begin{equation*}
    X_{k}\cdots X_{1}h(p)=\left\langle X_{k}(p), \nabla X_{k-1}\cdots X_{1}h(p)\right\rangle_p,
    \end{equation*}
that is, $X_{k}\cdots X_{1}h(p)$ is the Lie derivative of the smooth function $X_{k-1}\cdots X_{1}h$ in the direction of the vector field $X_{k}$ at $p$. In particular, $X^{k}h(p)$ denotes $X_{k}\cdots X_{1}h(p)$, where $X_{i}=X$, for $i=1,\cdots,k$.

We say that a tangency point $p \in \Switch$ of $F=(F_+,F_-)$ has \defemph{finite multiplicity} if there exist natural numbers $m_+, m_- \geq 2$ such that $F_+^{m_+} h(p)\neq 0$ and $F_-^{m_-} h(p)\neq 0$. In this case, there is only a finite number of regular trajectories of $F_+$ and $F_-$ that arrive at such tangency point or depart from it \cite{AGN}.

\begin{remark}
For our purposes, we assume that the vector fields $F_+$ and $F_-$ are bounded, the tangency set $\Tange$ is a finite set and every tangency point has finite multiplicity.
\end{remark}

The \defemph{Filippov solutions} of a piecewise smooth vector field $F$ are defined as follows. First, we construct a new vector field $F_s(p)$ at $p \in \Slide$, which is given by the convex combination of $F_+(p)$ and $F_-(p)$ that is tangent to $\Slide$. We refer to $F_s$ as the \defemph{sliding vector field} of $F$ on $\Slide$, and it is explicitly given by
    \begin{equation}
    F_s(p) := \frac{1}{F_-h(p) - F_+h(p)} (F_-h(p) F_+(p) - F_+h(p) F_-(p)).
    \end{equation}
Notice that $F_s$ is always well-defined since its denominator is always nonzero in $\Slide$. Also, in some cases $F_s$ can be extended to $\overline{\Slide}$. 

We denote the flow of $F_+$ for time $t$ as $\Phi_+^t$, the flow of $F_-$ for time $t$ as $\Phi_-^t$ and the flow of $F_s$ (over $\Slide$) for time $t$ as $\Phi_s^t$. When $p\in\Cross$, the local solution of $F$ is given by the concatenation of the solutions $\Phi_+^t(p)$ of $F_+$ in $M_+$ and $\Phi_-^t(p)$ of $F_-$ in $M_-$. Now, when $p\in\Slide\cup\Tange$, the local solution is given by any continuous piecewise smooth parameterized trajectory obtained by the concatenation of $p$, $\Phi_+^t(p)$ (restricted to $M_+$), $\Phi_-^t(p)$ (restricted to $M_-$) and $\Phi_s^t(p)$ (we also consider its extension to $\Tange$ when it is possible). In this way, the solutions of $F$ are all the continuous, piecewise smooth trajectories whose smooth pieces are integral trajectories of $F_+$, $F_-$ or $F_s$.
(When the manifold has boundary, this may include trajectories that reach the boundary of the space and stop there. In this case the trajectory is not defined for all time values.)
The pieces must be glued to each other on a point of $\Switch$. We also refer to these solutions as \defemph{orbits} of $F$. A maximal orbit is an orbit whose interval of definition cannot be enlarged.

Notice that this approach may give rise to a lack of uniqueness of solutions, since different trajectories on $M_+$, $\Slide$ and $M_-$ may pass through the same point, hence a flow $\Phi$ cannot be defined for every Filippov PSVF. In light of this, we introduce in \cref{sec:orbit_space} a new flow $\wt \Phi$ defined on a larger space $\wt M$ related to the PSVF, which will be able to restore the idea of uniqueness lost in this scenario.

A special attention must be paid to the singularities of a PSVF $F=(F_+,F_-)$. Since we consider a new way to define solutions, we must distinguish some points of $\Switch$ which will also behave as singularities in a certain way. A point $p \in \Switch$ is said to be a \defemph{$\Sigma$-singularity} of $F$ provided that $p$ is either a point of $\Tange$ such that $F_+(p), F_-(p)\neq 0$, an equilibrium of $F_+$ or $F_-$, or an equilibrium of $F_s$ (known as \defemph{pseudo-equilibrium} of $F$). A point $p \in \Switch$ which is not a $\Sigma$-singularity of $F$ is also referred as a \defemph{regular-regular} point of $F$. We say that $\gamma$ is a \defemph{regular orbit} of $F=(F_+,F_-)$ if it is a piecewise smooth curve such that $\gamma\cap M_{+}$ and $\gamma\cap M_{-}$ are unions of regular orbits of $F_+$ and $F_-$, respectively, and  $\gamma\cap\Switch\subset\Cross$.
More details on the classification of $\Sigma$-singularities can be found in \cite{GST-generic_bifurcations, Kuznetsov}

In this work we use the following classical conception of transitivity for continuous flows. Let $M$ be a metric space and $\Phi$ be a continuous flow on $M$. We say $\Phi$ is \defemph{topologically transitive} if, for every pair of non-empty open sets $U$ and $V$ in $M$, there is a strictly positive time $t > 0$ such that $\Phi^t(U) \cap V \neq \emptyset$. We say $\Phi$ is \defemph{transitive} if there exists an orbit of the system that is dense in $M$. These different concepts of transitivity are related in the following result (see \cite{grosse2011linear}).

\begin{theorem}[Birkhoff transitivity theorem]
\label{theo:birkhoff.transitivity}
Let $M$ be a perfect (has no isolated points), separable, Baire metric space and $\Phi$ a continuous flow on it. Then $\Phi$ is topologically transitive if, and only if it is transitive (has a dense orbit).
\end{theorem}

Finally, we say that a Filippov PSVF is \defemph{transitive} if there exists a Filippov orbit that is dense in the space.

\section{Orbit space}
\label{sec:orbit_space}

Let $M$ be a manifold (possibly with boundary), which will be the \defemph{phase space} of the system, and let $F$ be a bounded Filippov vector field on $M$. (Notice that $F$ being bounded is equivalent to both $F_+$ and $F_-$ being bounded, since from \cref{eq:definition.F} we have $\nor{F} \leq \nor{F_+} + \nor{F_-}$.) We denote the supremum norm as $\nor{F}$, and assume $\nor{F} > 0$. The \defemph{orbit space} of the system is the set of all maximal orbits of $F$, denoted $\wt M$. Using the PSVF given by $F$ we can define a flow on $\wt M$ by
    \begin{align*}
    \func{\wt \Phi^t}{\wt M}{\wt M}{\gamma}{
    \begin{aligned}[t]
        \func{\wt \Phi^t(\gamma)}{\R}{M}{s}{\gamma(t+s).}
    \end{aligned}
    }
    \end{align*}
(The flow may not actually be defined for every $\R$, since the trajectory $\gamma \in \wt M$ may not be defined for all time values. In this case the flow is defined only when the definition makes sense.) This results in a continuous dynamical system on the orbit space (\cref{prop:orbit_space_flow_is_continuous}). In order to understand this flow we need more structure on the orbit space $\wt M$, so we shall introduce the structure of a metric space for $\wt M$.

\subsection{The orbit distance function}

The distance function on $M$ is the distance induced from a Riemannian metric on $M$, and is denoted by $d\colon M \times M \to \R$. Using this distance function $d$, we define the \defemph{integral distance} $\wt d$ on the orbit space $\wt M$ by 
    \begin{align*}
        \func{\wt d}{\wt M \times \wt M}{\R}{(\gamma_0,\gamma_1)}{\wt d(\gamma_0, \gamma_1) := \sum_{i \in \Z} \frac{1}{2^{\abs{i}}} \int_{i}^{i+1} d(\gamma_0(t), \gamma_1(t)) \d t}.
    \end{align*}

We denote the \defemph{ball} of center $\gamma \in \wt M$ and radius $r \in \R_{\geq 0}$ for the distance $\wt d$ by $\wtbola{\gamma}{r}$.
Notice that in the definition of our metric we are assuming that our orbits are defined for every $t \in \R$. We will do so throughout the rest of the paper.

The orbit space with this distance function becomes a metric space (\cref{prop:orbit_space_is_metric}). This has already been shown for compact spaces \cite{ACV}, but here we assume only that the Filippov vector field is bounded. We must show that the sum in the definition of the distance is always finite and that it is in fact a distance. First we prove the following lemma, which assumes the vector field $F$ is bounded and shows as a consequence that specific suprema related to the distance of trajectory points must also be bounded.

\begin{lemma}
\label{lemma:supremum.majoration}
Let $M$ be a manifold and $F$ be a bounded Filippov vector field on $M$.
For every $\gamma_0, \gamma_1 \in \wt M$ and $n \in \Z$,
    \begin{equation}
    \label{eq:sup.bounded}
    \sup_{n \leq t < n+1} d(\gamma_0(t), \gamma_1(t)) \leq d(\gamma_0(0), \gamma_1(0)) + 2\nor{F}(1+|n|).
    \end{equation}
\end{lemma}
\begin{proof}
We will first prove that, for every $i \in \Z$,
    \begin{equation}
    \label{lemma:supremum.majoration.positive}
    \sup_{i \leq t < i+1} d(\gamma_0(t), \gamma_1(t)) \leq \sup_{i-1 \leq t < i} d(\gamma_0(t), \gamma_1(t)) + 2\nor{F}.
    \end{equation}
For every $t \in \left[i,i+1\right[$ it follows from the triangle inequality that
    \begin{equation*}
    d(\gamma_0(t), \gamma_1(t)) \leq d(\gamma_0(t), \gamma_0(i)) + d(\gamma_0(i), \gamma_1(i)) + d(\gamma_1(i), \gamma_1(t)).
    \end{equation*}
Since the vector field is bounded by $\nor{F}$, it follows from the mean value inequality that $d(\gamma_0(t), \gamma_0(i)) \leq \nor{F}|t-i|$ and $d(\gamma_1(i), \gamma_1(t)) \leq \nor{F}|i-t|$, hence
    \begin{align*}
    \sup_{i \leq t < i+1} d(\gamma_0(t), \gamma_1(t)) &\leq \sup_{i \leq t < i+1} d(\gamma_0(t), \gamma_0(i)) + d(\gamma_0(i), \gamma_1(i)) + \sup_{i \leq t < i+1} d(\gamma_1(i), \gamma_1(t)) \\
        &\leq \sup_{i \leq t < i+1} \nor{F}|t-i| + d(\gamma_0(i), \gamma_1(i)) + \sup_{i \leq t < i+1} \nor{F}|i-t| \\
        &\leq \sup_{i-1 \leq t < i} d(\gamma_0(t), \gamma_1(t)) + 2\nor{F}.
    \end{align*}
Analogously, for every $i \in \Z$ we can also obtain that
    \begin{equation}
    \label{lemma:supremum.majoration.negative}
    \sup_{i-1 \leq t < i} d(\gamma_0(t), \gamma_1(t)) \leq \sup_{i \leq t < i+1} d(\gamma_0(t), \gamma_1(t)) + 2\nor{F}.
    \end{equation}

By induction on $n$, using \cref{lemma:supremum.majoration.positive} for positive $n$ and \cref{lemma:supremum.majoration.negative} for negative $n$, we conclude that
    \begin{equation}
    \label{eq:sup.bounded.by.sup}
    \sup_{n \leq t < n+1} d(\gamma_0(t), \gamma_1(t)) \leq \sup_{0 \leq t < 1} d(\gamma_0(t), \gamma_1(t)) + 2\nor{F}\abs{n}.
    \end{equation}

Finally, for every $t \in \left[0,1\right]$,
    \begin{align*}
    d(\gamma_0(t), \gamma_1(t)) &\leq d(\gamma_0(t), \gamma_0(0)) + d(\gamma_0(0), \gamma_1(0)) + d(\gamma_1(0), \gamma_1(t)) \\
        &\leq \nor{F}\abs{-t} + d(\gamma_0(0), \gamma_1(0)) + \nor{F}\abs{t} \\
        &= d(\gamma_0(0), \gamma_1(0)) + 2\nor{F}\abs{t},
    \end{align*}
so $\sup_{0 \leq t < 1}  d(\gamma_0(t), \gamma_1(t)) \leq d(\gamma_0(0), \gamma_1(0)) + 2\nor{F}$. From this equation and \cref{eq:sup.bounded.by.sup} we obtain \cref{eq:sup.bounded}.
\end{proof}

\begin{proposition}
\label{prop:orbit_space_is_metric}
Let $M$ be a manifold and $F$ be a bounded Filippov system on $M$.
The function $\wt d$ is a distance function on $\wt M$.
\end{proposition}
\begin{proof}
The most important part is to prove the function is well-defined in the sense that the infinite sum always has a finite value. For this we will use the fact that the vector field is bounded. Let $\gamma_0, \gamma_1 \in \wt M$ and $d_0 := d(\gamma_0(0), \gamma_1(0))$. From \cref{lemma:supremum.majoration} it follows (using that $\sum_{i=n}^{\infty} \frac{i}{2^i} = \frac{n+1}{2^{n-1}}$) that
    \begin{align*}
    d(\gamma_0, \gamma_1) &= \sum_{i \in \Z} \frac{1}{2^{\abs{i}}} \int_{i}^{i+1} d(\gamma_0(t), \gamma_1(t)) \d t \\
        &\leq \sum_{i \in \Z} \frac{1}{2^{\abs{i}}} \sup_{i \leq t < i+1} d(\gamma_0(t), \gamma_1(t)) \\
        &\leq \sum_{i \in \Z} \frac{d_0 + 2\nor{F}(1+\abs{i})}{2^{\abs{i}}} \\
        &= \sum_{i \in \Z} \frac{d_0+2\nor{F}}{2^{\abs{i}}} + 2\sum_{i \geq 1} \frac{2\nor{F}i}{2^{i}} \\
        &= 3(d_0+2\nor{F}) + 4\nor{F}\sum_{i \geq 1} \frac{i}{2^{i}} \\
        &= 3d_0 + 14\nor{F} < \infty.
    \end{align*}
This shows that $\wt d$ is well-defined.

Now we show the properties of a distance function. If $\gamma_0 = \gamma_1$, then $\wt d(\gamma_0, \gamma_1) = \sum_{i \in \Z} \frac{1}{2^{\abs{i}}} 0 = 0$; if $\wt d(\gamma_0, \gamma_1) = 0$, then $\int_{i}^{i+1} d(\gamma_0(t), \gamma_1(t)) \d t = 0$ for every $i \in \Z$, hence $\gamma_0(t) = \gamma_1(t)$ for every $t \in \left[i, i+1\right[$, therefore $\gamma_0 = \gamma_1$. Finally, the symmetry and the triangle inequality of $\wt d$ follow directly from the respective properties of $d$.
\end{proof}

The summation of suprema that appeared in the preceding proposition (in order to prove that the distance is finite) motivates the definition of another distance function on $\wt M$, the \defemph{supremum distance}
    \begin{align*}
        \func{\wt d_{\sup}}{\wt M \times \wt M}{\R}{(\gamma_0, \gamma_1)}{\wt d_{\sup}(\gamma_0, \gamma_1) := \sum_{i \in \Z} \frac{1}{2^{\abs{i}}} \sup_{i \leq t < i+1} d(\gamma_0(t), \gamma_1(t))}.
    \end{align*}

This function can be proven to be a distance function in the same way that was done for $\wt d$. These distance functions are topologically equivalent, as the following proposition shows, and thus will be used interchangeably when analysing topological properties of the orbit space.

\begin{proposition}
\label{prop:distances_are_topologically_equivalent}
Let $M$ be a manifold and $F$ be a bounded Filippov vector field on $M$.
The distance functions $\wt d$ e $\wt d_{\sup}$ are topologically equivalent.
\end{proposition}
\begin{proof}
Let us denote the balls of center $\gamma \in \wt M$ and radius $r>0$ for the distances $\wt d$ and $\wt d_{\sup}$ by $\wtbola{\gamma}{r}$ and $\wtbola[_{\sup}]{\gamma}{r}$, respectively, and their topologies by $\mathcal{T}$ and $\mathcal{T}_{\sup}$. We will prove that each topology is finer than the other.

($\mathcal{T} \subseteq \mathcal{T}_{\sup}$) For every $\gamma_0, \gamma_1 \in \wt M$
    \begin{equation*}
    	\begin{split}
    	\wt d(\gamma_0, \gamma_1) &= \sum_{i \in \Z} \frac{1}{2^{\abs{i}}} \int_{i}^{i+1} d(\gamma_0(t), \gamma_1(t)) \d t  \\
    		&\leq \sum_{i \in \Z} \frac{1}{2^{\abs{i}}} \sup_{i \leq t < i+1} d(\gamma_0(t), \gamma_1(t)) = \wt d_{\sup}(\gamma_0, \gamma_1).
    	\end{split}
        \end{equation*}
This implies that every ball of $\wt d_{\sup}$ is contained in the ball of $\wt d$ with same center and radius, hence that the topology generated by $\wt d_{\sup}$ is finer than the one generated by $\wt d$.

($\mathcal{T}_{\sup} \subseteq \mathcal{T}$) Take $\gamma \in \wt M$ and $r' > 0$ and consider the ball $\wtbola[_{\sup}]{\gamma}{r'}$ with center point $\gamma$ and radius $r'$. We must find $r > 0$ such that $\wtbola{\gamma}{r} \subseteq \wtbola[_{\sup}]{\gamma}{r'}$. Suppose, for the sake of contradiction, that such $r$ did not exist. In that case, there would exist a sequence $(r_n)_{n \in \N}$ of positive real numbers such that $r_n \to 0$ and, for every $n \in \N$, an orbit $\gamma_n \in \wt M$ such that $\gamma_n \in \wtbola{\gamma}{r_n}$ and $\gamma_n \notin \wtbola[_{\sup}]{\gamma}{r'}$, which means that
    \begin{equation}
    \label{eq:metrica.integral.menor}
    \wt d(\gamma, \gamma_n) = \sum_{i \in \Z} \frac{1}{2^{\abs{i}}} \int_{i}^{i+1} d(\gamma(t), \gamma_n(t)) \d t < r_n,
    \end{equation}
and
    \begin{equation}
    \label{eq:metrica.supremo.maior}
     \wt d_{\sup}(\gamma, \gamma_n) = \sum_{i \in \Z} \frac{1}{2^{\abs{i}}} \sup_{i \leq t < i+1} d(\gamma(t), \gamma_n(t)) \geq r'.
    \end{equation}
From this it would follow that, for every $i \in \Z$, each term $\int_{i}^{i+1} d(\gamma(t), \gamma_n(t)) \d t$ of the summation in \cref{eq:metrica.integral.menor} would converge to $0$ as $n \to \infty$. Therefore, by continuity of the orbits in $\wt M$, each term $\sup_{i \leq t < i+1} d(\gamma(t), \gamma_n(t))$ of the summation in \cref{eq:metrica.supremo.maior} would also converge to $0$ as $n \to \infty$, while the summation in \cref{eq:metrica.supremo.maior} would be bounded below by $r'>0$. This would lead to the following contradiction.

Since the vector field is bounded by $\nor{F}$, it would follow from \cref{lemma:supremum.majoration} that, for every $n \in \N$ and every $i_0 \in \N$ big enough,
    \begin{equation*}
    \sum_{|i| > i_0} \frac{1}{2^{\abs{i}}} \sup_{i \leq t < i+1} d(\gamma(t), \gamma_n(t)) \leq \frac{d(\gamma(0), \gamma_n(0)) + 2\nor{F}(i_0+3)}{2^{i_0 - 1}} \leq \frac{r'}{2},
    \end{equation*}
so that
    \begin{equation*}
    \sum_{|i| \leq i_0} \frac{1}{2^{\abs{i}}} \sup_{i \leq t < i+1} d(\gamma(t), \gamma_n(t)) \geq \frac{r'}{2} > 0.
    \end{equation*}
But, for every $|i| \leq i_0$, $\sup_{i \leq t < i+1} d(\gamma(t), \gamma_n(t)) \to 0$ as $n \to \infty$, so there would be $n_0 \in \N$ big enough such that, for every $n \geq n_0$ and every $|i| \leq i_0$,
    \begin{equation*}
    \sup_{i \leq t < i+1} d(\gamma(t), \gamma_n(t)) < \frac{r'}{2(2i_0+1)},
    \end{equation*}
and so
    \begin{equation*}
    \sum_{|i| \leq i_0} \frac{1}{2^{\abs{i}}} \sup_{i \leq t < i+1} d(\gamma(t), \gamma_n(t)) < (2i_0+1)\frac{r'}{2(2i_0+1)} = \frac{r'}{2}.
    \end{equation*}

This contradiction shows that $\gamma' \in \wtbola[_{\sup}]{\gamma}{r'}$, hence that the topology generated by $\wt d$ is finer than the one generated by $\wt d_{\sup}$.
\end{proof}

Besides topological equivalence, if we assume the piecewise smooth vector field $F$ is Lipschitz continuous (which is true when $M$ is compact), then we can also show that $\wt d$ and $\wt d_{\sup}$ are Lipschitz equivalent, hence uniformly equivalent, which is stronger than topological equivalence. This is relevant when we calculate the topological entropy of our system, such as in \cite{ACV, EMV}, since in this case the value of the entropy is the same whether we use $\wt d$ and $\wt d_{\sup}$.

\begin{remark}
If our PSVF has Lipschitz constant $K \geq 0$ and we drop the boundedness condition, then in order for $\wt d$ and $\wt d_{\sup}$ to have finite values and thus be well-defined we must redefine them, substituting the constant $\frac{1}{2}$ in the series with any $\alpha > e^K$. We will not prove this here to avoid repetition, but the calculations follow methods analogous to those of \cref{lemma:supremum.majoration,prop:orbit_space_is_metric}.
\end{remark}

\begin{proposition}
Let $M$ be a manifold and $F$ be a Lipschitz continuous Filippov vector field on $M$.
The distance functions $\wt d$ e $\wt d_{\sup}$ are Lipschitz equivalent.
\end{proposition}
\begin{proof}
We have already shown in \cref{prop:distances_are_topologically_equivalent} that $\wt d \leq \wt d_{\sup}$. Let $K \geq 0$ be a Lipschitz constant for $F$.

Given $\gamma_0, \gamma_1 \in \wt M$ and $i \in \Z$, let $t_i \in [i, i+1]$ be such that
	\begin{equation*}
	\sup_{i \leq t < i+1} d(\gamma_0(t), \gamma_1(t)) = d(\gamma_0(t_i), \gamma_1(t_i)).
	\end{equation*}
Since $F$ is Lipschitz continuous, it follows (cf. \cite[\S 9.2]{Coleman}) that, for every $t \in [i, i+1]$,
	\begin{equation*}
	d(\gamma_0(t_i), \gamma_1(t_i)) \leq e^K d(\gamma_0(t), \gamma_1(t)),
	\end{equation*}
which implies that
	\begin{equation*}
	\sup_{i \leq t < i+1} d(\gamma_0(t), \gamma_1(t)) = \int_{i}^{i+1} \sup_{i \leq t < i+1} d(\gamma_0(t), \gamma_1(t)) \d t \leq \int_{i}^{i+1} e^K d(\gamma_0(t), \gamma_1(t)) \d t.
	\end{equation*}
Therefore we obtain that
	\begin{equation*}
	\begin{split}
	\wt d_{\sup}(\gamma_0, \gamma_1) &= \sum_{i \in \Z} \frac{1}{2^{\abs{i}}} \sup_{i \leq t < i+1} d(\gamma_0(t), \gamma_1(t)) \\
		&\leq \sum_{i \in \Z} \frac{1}{2^{\abs{i}}} \int_{i}^{i+1} e^K d(\gamma_0(t), \gamma_1(t)) \d t = e^K \wt d(\gamma_0, \gamma_1).
	\end{split}	
	\end{equation*}
This shows that $\wt d_{\sup} \leq e^K \wt d$, therefore $\wt d$ e $\wt d_{\sup}$ are Lipschitz equivalent.
\end{proof}

Throughout the rest of this paper we do not assume that $F$ is Lipschitz continuous.

\subsection{Properties of the orbit distance function}
\label{ssec:properties_orbit_distance}

It is very important to understand the intuitive meaning of the metric $\wt d$ (or $\wt d_{\sup}$). Two orbits being close in the orbit space $\wt M$ should be the same as the points of these orbits being sufficiently close for a sufficient amount of time in the phase space $M$ (see \cref{fig:3.4e6}). The following \cref{lemma:close.orbit.points.imply.close.orbits,lemma:close.orbits.imply.close.orbit.points} provide a precise meaning of these ideas. We first show that if the points of two orbits are close for enough time in the phase space, then the orbits are close in the orbit space.

\begin{lemma}
\label{lemma:close.orbit.points.imply.close.orbits}
Let $M$ be a manifold and $F$ be a bounded Filippov system on $M$.
Given $\varepsilon > 0$, there exist $\tau > 0$ and $\delta > 0$ such that, for every $\gamma_0, \gamma_1 \in \wt M$, if $d(\gamma_0(t),\gamma_1(t)) < \delta$ for every $t \in \left]-\tau, \tau\right[$, then $\wt d(\gamma_0,\gamma_1) < \varepsilon$.
\end{lemma}
\begin{proof}
Let $\varepsilon > 0$, choose an integer $\tau \geq 2$ such that $\nor{F}\frac{\tau+2}{2^{\tau-3}} < \varepsilon$, and define $\delta := \frac{1}{3}\left(\varepsilon - \nor{F}\frac{\tau+2}{2^{\tau-3}} \right) > 0$. Take $\gamma_0, \gamma_1 \in \wt M$. If $d(\gamma_0(t),\gamma_1(t)) < \delta$ for every $t \in \left]-\tau, \tau\right[$, then it follows from \cref{lemma:supremum.majoration} that $\sup_{i \leq t < i+1} d(\gamma_0(t), \gamma_1(t)) \leq \delta + 2\nor{F}(1+\abs{i})$, so (using that $\sum_{i=n}^{\infty} \frac{i}{2^i} = \frac{n+1}{2^{n-1}}$) 
	\begin{equation*}
	\begin{split}
	\wt d(\gamma_0, \gamma_1) &\leq \sum_{i \in \Z} \frac{1}{2^{\abs{i}}} \sup_{i \leq t < i+1} d(\gamma_0(t), \gamma_1(t)) \\
		&= \sum_{|i| < \tau-1} \frac{1}{2^{\abs{i}}} \sup_{i \leq t < i+1} d(\gamma_0(t), \gamma_1(t)) + \sum_{|i| \geq \tau} \frac{1}{2^{\abs{i}}} \sup_{i \leq t < i+1} d(\gamma_0(t), \gamma_1(t)) \\
		&\leq \sum_{|i| < \tau-1} \frac{\delta}{2^{\abs{i}}} + \sum_{|i| \geq \tau} \frac{\delta + 2\nor{F}(1+\abs{i})}{2^{\abs{i}}} \\
	    &= \sum_{i \in \Z} \frac{\delta}{2^{\abs{i}}} + 4\nor{F}\sum_{i = \tau}^{\infty} \frac{1+\abs{i}}{2^{\abs{i}}} \\
         &= \sum_{i \in \Z} \frac{\delta}{2^{\abs{i}}} + 4\nor{F}\sum_{i = \tau}^{\infty} \frac{1+\abs{i}}{2^{\abs{i}}} \\
		&= 3\delta + \nor{F}\frac{\tau+2}{2^{\tau-3}} \\
		&= \varepsilon.
		\qedhere
	\end{split}    
	\end{equation*}
\end{proof}

\begin{figure}
    \centering
	\includegraphics{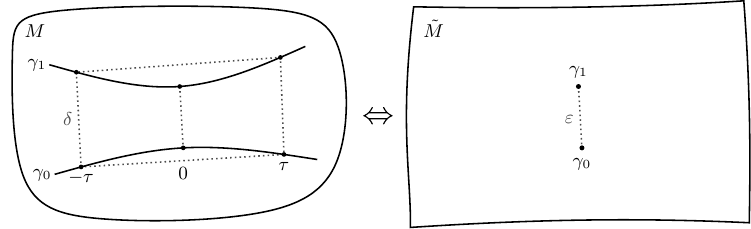}
    \caption{Graphical representation of \cref{lemma:close.orbit.points.imply.close.orbits,lemma:close.orbits.imply.close.orbit.points}.}
    \label{fig:3.4e6}
\end{figure}

To show that if two orbits are close in the orbit space, then the points of the orbits are close for enough time in the phase space, we must first prove the following lemma (see \cref{fig:3.5}).

\begin{lemma}
\label{lemma:approximation.orbits}
Let $M$ be a manifold and $F$ be a bounded Filippov system on $M$.
Let $\delta > 0$, $\alpha > 0$, $\gamma_0, \gamma_1 \in \wt M$ and $t_0 \in \R$. If $d(\gamma_0(t_0), \gamma_1(t_0)) \geq \delta$, then, for every $t \in \left[t_0 - \frac{\alpha}{2\nor{F}}, t_0 + \frac{\alpha}{2\nor{F}} \right]$,
    \begin{equation*}
    d(\gamma_0(t), \gamma_1(t)) \geq \delta - \alpha.
    \end{equation*}
\end{lemma}
\begin{proof}
Since $F$ is bounded by $\nor{F}$, it follows from the mean value inequality that, for every $\gamma \in \wt M$ and every $t, t' \in \R$,
    \begin{equation*}
    d(\gamma(t), \gamma(t')) \leq \nor{F}\abs{t'-t}.
    \end{equation*}
Then, for every $\gamma_0, \gamma_1 \in \wt M$ and every $t \in \left[t_0 - \frac{\alpha}{2\nor{F}}, t_0 + \frac{\alpha}{2\nor{F}} \right]$,
    \begin{align*}
    \delta &\leq d(\gamma_0(t_0), \gamma_1(t_0)) \\
        &\leq d(\gamma_0(t_0), \gamma_0(t)) + d(\gamma_0(t), \gamma_1(t)) +  d(\gamma_1(t), \gamma_1(t_0)) \\
        &\leq \nor{F}\abs{t-t_0} + d(\gamma_0(t), \gamma_1(t)) + \nor{F}\abs{t_0-t} \\
        &\leq 2\nor{F}\frac{\alpha}{2\nor{F}} + d(\gamma_0(t), \gamma_1(t)),
    \end{align*}
therefore $\delta - \alpha \leq d(\gamma_0(t), \gamma_1(t))$.
\end{proof}

\begin{figure}
    \centering
	\includegraphics{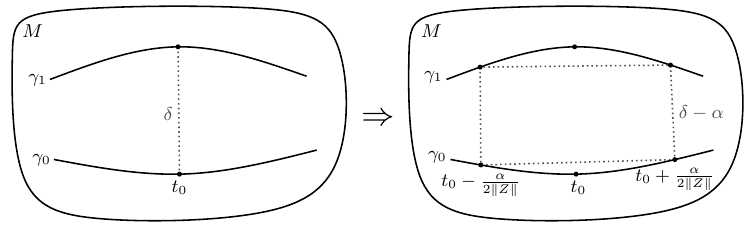}
    \caption{Graphical representation of \cref{lemma:approximation.orbits}.}
    \label{fig:3.5}
\end{figure}

The next result has also appeared in \cite[Prop. 4.2]{ACV}; we state it and also prove it here since we understand that it is important for the full comprehension of the ideas we are presenting.

\begin{lemma}
\label{lemma:close.orbits.imply.close.orbit.points}
Let $M$ be a manifold and $F$ be a bounded Filippov system on $M$.
For every $\tau > 0$ and $\delta > 0$, there exists $\varepsilon > 0$ such that, for every $\gamma_0, \gamma_1 \in \wt M$, if $\wt d(\gamma_0, \gamma_1) < \varepsilon$
then $d(\gamma_0(t), \gamma_1(t)) < \delta$ for every $t \in \left]-\tau, \tau \right[$.
\end{lemma}
\begin{proof}
We will prove this by contradiction. Suppose some $\tau > 0$ and some $\delta > 0$ satisfy that, for every $\varepsilon > 0$, there are $\gamma_0^\varepsilon, \gamma_1^\varepsilon \in \wt M$ and a time $t_\varepsilon \in \left]-\tau, \tau \right[$ such that $\wt d(\gamma_0^\varepsilon, \gamma_1^\varepsilon) < \varepsilon$ and $d(\gamma_0^\varepsilon(t_\varepsilon), \gamma_1^\varepsilon(t_\varepsilon)) \geq \delta$. Taking $\alpha = \frac{\delta}{2}$ in \cref{lemma:approximation.orbits}, and defining $I_\varepsilon := \left[t_\varepsilon - \frac{\delta}{4\nor{F}}, t_\varepsilon + \frac{\delta}{4\nor{F}}\right]$, it follows that, for every $t \in I_\varepsilon$,
    \begin{equation}
    d(\gamma_0^\varepsilon(t), \gamma_1^\varepsilon(t)) \geq \delta - \frac{\delta}{2} = \frac{\delta}{2} > 0.
    \end{equation}
Notice that the size of $I_\varepsilon$ is independent of $\varepsilon$, since $|I_\varepsilon| = \frac{\delta}{2\nor{F}}$, and that $I_\varepsilon \subseteq \left[-\tau-\frac{\delta}{4\nor{F}}, \tau+\frac{\delta}{4\nor{F}} \right]$. Therefore it follows that, for every $\varepsilon > 0$,
    \begin{equation*}
    \varepsilon > \wt d(\gamma_0^\varepsilon, \gamma_1^\varepsilon) > \frac{1}{2^{\lceil\tau+\frac{\delta}{4\nor{F}}\rceil}} \int_{I_\varepsilon} d(\gamma_0^\varepsilon(t), \gamma_1^\varepsilon(t)) \d t \geq \frac{1}{2^{\lceil\tau+\frac{\delta}{4\nor{F}}\rceil}} \delta|I_\varepsilon| = \frac{\delta^2}{2^{\lceil\tau+\frac{\delta}{4\nor{F}}\rceil} 2\nor{F}},
    \end{equation*}
So choosing $\varepsilon \leq \frac{\delta^2}{2^{\lceil\tau+\frac{\delta}{4\nor{F}}\rceil} 2\nor{F}}$ leads to a contradiction.
\end{proof}

The first use we make of \cref{lemma:close.orbit.points.imply.close.orbits,lemma:close.orbits.imply.close.orbit.points} is to prove the continuity of $\wt \Phi$.

\begin{proposition}
\label{prop:orbit_space_flow_is_continuous}
Let $M$ be a manifold and $F$ a bounded Filippov vector field on $M$. The flow $\wt \Phi$ is continuous.
\end{proposition}
\begin{proof}
We are going to assume the domain of the flow $\wt \Phi$ is $\R \times \wt M$ and also assume that the distance function on the product space $\R \times \wt M$ is given by the maximum of the distances on $\R$ and $\wt M$.

Take $(s_0, \gamma_0) \in \R \times \wt M$ and $\varepsilon > 0$. We must find $\delta > 0$ such that, for every $(s,\gamma) \in \R \times \wt M$, if $\abs{s-s_0} < \delta$ and $\wt d(\gamma_0,\gamma)<\delta$, then $\wt d(\wt\Phi^{s_0}(\gamma_0), \wt\Phi^{s}(\gamma)) < \varepsilon$.

First notice that, since $F$ is bounded by $\nor{F}$, it follows from the mean value inequality that, for every $t \in \R$,
	\begin{equation*}
	\begin{split}
	d(\gamma_0(t+s_0), \gamma(t+s)) &\leq d(\gamma_0(t+s_0), \gamma(t+s_0)) + d(\gamma(t+s_0), \gamma(t+s)) \\
		&\leq d(\gamma_0(t+s_0), \gamma(t+s_0)) + \nor{F}\abs{s-s_0}.
	\end{split}
	\end{equation*}
Then, for every $i \in \Z$,
    \begin{equation*}
    \sup_{i \leq t < i+1} d(\gamma_0(t+s_0), \gamma(t+s)) \leq \sup_{i \leq t < i+1}  d(\gamma_0(t+s_0), \gamma(t+s_0)) + \nor{F}\abs{s-s_0},
    \end{equation*}
so, by summing over all integers $i \in \Z$ with the corresponding weights $2^{-|i|}$,

    \begin{equation}
    \label{eq:continuity.flow}
    \wt d(\wt\Phi^{s_0}(\gamma_0), \wt\Phi^{s}(\gamma)) = \wt d(\wt\Phi^{s_0}(\gamma_0), \wt\Phi^{s_0}(\gamma)) + 3\nor{F}\abs{s-s_0}.
    \end{equation}

Now notice that, for every $\varepsilon' > 0$, there is a $\delta' > 0$ such that, if $\wt d(\gamma_0, \gamma) < \delta'$, then $\wt d(\wt\Phi^{s_0}(\gamma_0), \wt\Phi^{s_0}(\gamma)) < \varepsilon'$. This is the case since, from \cref{lemma:close.orbit.points.imply.close.orbits}, there are $\tau_0 > 0$ and $\delta_0 > 0$ such that, if $d(\wt \Phi^{s_0}(\gamma_0)(t), \wt \Phi^{s_0}(\gamma)(t)) < \delta_0$ for every $t \in \left[-\tau_0, \tau_0 \right]$, then $\wt d(\wt\Phi^{s_0}(\gamma_0), \wt\Phi^{s_0}(\gamma)) < \varepsilon'$. But this hypothesis is equivalent to having $d(\gamma_0(t), \gamma(t)) < \delta_0$ for every $t \in \left[-\tau_0+s_0, \tau_0+s_0 \right]$. So by choosing some $\tau_1 > 0$ such that $\left[-\tau_0+s_0, \tau_0+s_0 \right] \subseteq \left[-\tau_1, \tau_1 \right]$, it follows from \cref{lemma:close.orbits.imply.close.orbit.points} that there is a $\delta' > 0$ such that, if $\wt d(\gamma_0, \gamma) < \delta'$, then $d(\gamma_0(t), \gamma(t)) < \delta_0$ for every $t \in \left[-\tau_1, \tau_1 \right]$.

Taking $\varepsilon' := \frac{\varepsilon}{2}$, using the respective $\delta'$ of the last paragraph, and defining $\delta := \min\{\delta', \frac{\varepsilon}{6\nor{F}}\}$, we conclude from \cref{eq:continuity.flow} that, if $\abs{s-s_0} < \delta$ and $\wt d(\gamma_0,\gamma)<\delta$, then
    \begin{equation*}
    \wt d(\wt\Phi^{s_0}(\gamma_0), \wt\Phi^{s}(\gamma)) < \frac{\varepsilon}{2} + 3\nor{F}\frac{\varepsilon}{6\nor{F}} = \varepsilon.
    \qedhere
    \end{equation*}
\end{proof}

\begin{proof}[Proof of {\cref{theo:orbit_space_is_metric_and_flow_is_continuous}}]
This is just \cref{prop:orbit_space_is_metric,prop:orbit_space_flow_is_continuous}.
\end{proof}

\subsection{Topology on the orbit space}
\label{ssec:topology_orbit_space}

We now study the topological structure of the orbit space. Our manifold $M$ is a metric space that is separable (has a countable dense subset), complete and perfect (has no isolated points). We show that the orbit space $\wt M$ inherits separability (\cref{prop:orbit.space.is.separable}) and being perfect (\cref{prop:orbit.space.isolated.points}), but completeness implies only that the orbit space is Baire (\cref{prop:orbit.space.is.Baire}).

\begin{proposition}
\label{prop:orbit.space.is.separable}
Let $M$ be a manifold and $F$ be a bounded Filippov system on $M$.
If $\Tange$ is finite and each tangency point has finite multiplicity, then
the orbit space $\wt M$ is separable.
\end{proposition}
\begin{proof}
The proof is divided in $5$ steps.
\paragraph{\bfseries Step 1.} (Augmenting the countable dense subset) The first step of the proof is to augment a countable dense subset of $M$ to work better with respect to tangencies. For each tangency point $p \in \Tange$ (see \cref{fig:3.8-1}), there are (up to) four orbit segments (this is a consequence of the finite multiplicity of the tangency points) starting at $p$: orbit segment $\gamma_+^+\colon \left[0, t_+^+ \right] \to M$ goes forwards from $p$, leaving $\Slide[u] \cup \Tange$ at $\gamma_+^+(0) = p$ to the region $M_+$, and stops the first time the orbit reaches $\Slide[s] \cup \Tange$ after that (if it does not, we take $t_+^+ = \infty$ and consider the interval $\left[0,+\infty \right[$ as the domain); orbit segment  $\gamma_+^-\colon \left[0, t_+^- \right] \to M$ is defined analogously, but leaves at $p$ to region $M_-$; orbit segment $\gamma_-^+\colon \left[-t_-^+, 0\right] \to M$ goes backwards from $p$, leaving $\Slide[s] \cup \Tange$ at $\gamma_-^+(0) = p$ to region $M_+$, and stops the first time it reaches $\Slide[u] \cup \Tange$ after that; and finally orbit segment $\gamma_-^-\colon \left[-t_-^-, 0\right] \to M$ is defined analogously to $\gamma_-^+$, but leaving $p$ to region $M_-$. The image of each of these orbit segments is in a regular region, apart possibly from their endpoints, so they are homeomorphic to a real interval and thus we can take a countable dense subset of them. The union of all these sets is also a countable set, which we denote as $E^\mathrm{t}_p$, and finally the set $E^\mathrm{t} := \bigcup_{p \in \Tange} E^\mathrm{t}_p$ is also a countable set, since $\Tange$ is finite. Finally, let $E_0$ be a countable dense subset of $M$ and define $E_1:= E_0 \cup E^\mathrm{t}$.

\begin{figure}[H]
    \centering
	\includegraphics{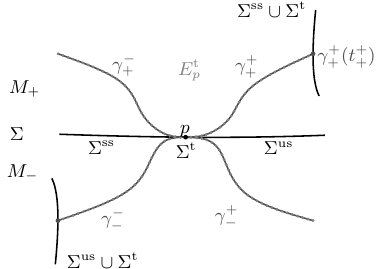}
    \caption{
    A dense set of points $E^\mathrm{t}$ is taken on the orbit segments arriving and leaving each tangency point $p$.}
    \label{fig:3.8-1}
\end{figure}

\paragraph{\bfseries Step 2.}(Defining a dense subset of $\Slide \cup \Tange$) For each $p \in E_1$, we define points $p^\mathrm{ss} \in \Slide[s] \cup \Tange$ and $p^\mathrm{us} \in \Slide[u] \cup \Tange$ as follows (see \cref{fig:3.8-2}). We consider an orbit $\gamma$ of $F$ such that $\gamma(0)=p$ and take $p^\mathrm{ss} := \gamma(t^\mathrm{ss})$, where $t^\mathrm{ss} \geq 0$ is the smallest positive time $t$ such that $\gamma(t) \in \Slide[s] \cup \Tange$ (that is, the first point in the orbit that reaches $\Slide[s] \cup \Tange$ going forward); if the orbit never enters $\Slide[s] \cup \Tange$, this point is left undefined. This is not dependent on the choice of $\gamma$ since, before entering $\Slide[s] \cup \Tange$ for the first time, the orbit of $p$ is unique, hence all such orbits coincide. Likewise, we take $p^\mathrm{us} := \gamma(-t^\mathrm{us})$, where $t^\mathrm{us} \geq 0$ is the smallest positive time $t$ such that $\gamma(-t) \in \Slide[u] \cup \Tange$ (that is, the first point in the orbit that reaches $\Slide[u] \cup \Tange$ going backward); if the orbit never reaches $\Slide[u] \cup \Tange$, this point is left undefined. We define sets $E^\mathrm{ss} := \set{p^\mathrm{ss} \st p \in E_1}$ and $E^\mathrm{us} := \set{p^\mathrm{us} \st p \in E_1}$. As a consequence of the definition of $E^\mathrm{t}$, all tangency points belong to these sets.
    
Notice that the set $E^\mathrm{ss}$ is dense in $\Slide[s] \cup \Tange$ and the set $E^\mathrm{us}$ is dense in $\Slide[u] \cup \Tange$. If it were otherwise, there would be a neighbourhood of a point $p \in \Slide[s]$ (respectively $\Slide[u]$) without any points of $E^\mathrm{ss}$ (respectively $E^\mathrm{us}$), which could be translated backwards (resp. forwards) along the orbits to create a tubular neighborhood of a point of $M$ without any point of $E_1$, which would contradict the fact that $E_1$ is dense in $M$. Finally, define $E := E_1 \cup E^\mathrm{ss} \cup E^\mathrm{us}$.

\begin{figure}[H]
    \centering
	\includegraphics{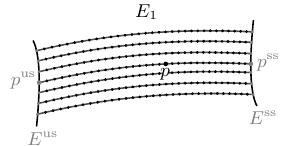}
    \caption{
    The sets of points $E^\mathrm{ss}$ and $E^\mathrm{us}$ (dense on $\Slide[s]$ and $\Slide[u]$, respectively) are taken by flowing the points $p \in E_1$ along their orbits until they reach the sliding region at the points $p^\mathrm{ss}$ and $p^\mathrm{us}$.}
    \label{fig:3.8-2}
\end{figure}

\paragraph{\bfseries Step 3.}
(Classifying orbit behavior through $\Slide \cup \Tange$) For each orbit $\gamma \in \wt M$ and each $\tau \geq 0$, we will define a finite sequence associated with $\gamma$ that describes how it travels through $\Switch$ in the time interval $\left[-\tau, \tau\right]$. Going forward (see \cref{fig:3.8-3.1}), in the time interval $\left[-\tau, \tau\right]$, $\gamma$ may enter and leave $\Switch$ only a finite number of times. Let $k^+ \in \N \cup \{0\}$ be the number of times it leaves $\Slide[u] \cup \Tange$. For each $i \in \{1, \ldots, k^+\}$, let $a_i \in \Slide[u] \cup \Tange$ be the point through which $\gamma$ leaves $\Slide[u] \cup \Tange$ for the $i$-th time. When the orbit leaves $\Slide[u] \cup \Tange$, it must go to either of the $2$ regions $M_+$ or $M_-$; let $\sigma_i := +1$ in the first case and $\sigma_i := -1$ in the second. In this way we define a finite sequence with $k^+$ entries:
    \begin{equation*}
    (a_1,\sigma_1), \ldots, (a_{k^+},\sigma_{k^+}).
    \end{equation*}
    In the case $k^+ = 0$ we just take the empty sequence to represent the orbits travel through $\Switch$.
    
Analogously, we construct a sequence for the orbit going backwards (see \cref{fig:3.8-3.2}), now taking $k^- \in \N$ to be number of times it enters $\Slide[s] \cup \Tange$ in the negative time interval $\left[-\tau, 0\right]$, $a_{-i} \in \Slide[s] \cup \Tange$ to be the point this happens at for the $i$-th time going backwards, and $\sigma_{-i} \in \{+1, -1\}$ to represent the region $M_+$ or $M_-$ it came from. This gives a finite sequence with $k^-$ entries:
    \begin{equation*}
    (a_{-k^-},\sigma_{-k^-}), \ldots, (a_{-1},\sigma_{-1}).
    \end{equation*}
The whole sequence is going to be called the $\Switch$-sequence of $\gamma$ in $\left[-\tau, \tau\right]$.

\begin{figure}[H]
    \centering
	\includegraphics{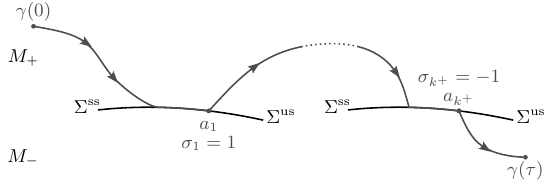}
    \caption{
    The $\Switch$-sequence of an orbit segment going forwards is defined by taking the points $a_1, \ldots, a_{k^+}$ where the orbit leaves $\Slide[u]$ on positive time. The values $\sigma_i$ give the side they left to.}
    \label{fig:3.8-3.1}
\end{figure}

\begin{figure}[H]
    \centering
	\includegraphics{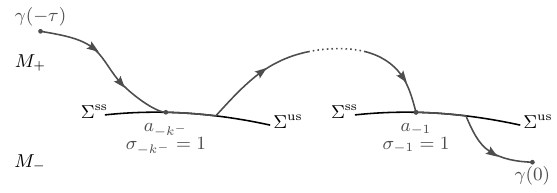}
    \caption{
    The $\Switch$-sequence of an orbit segment going backwards is defined by taking the points $a_{-k^-}, \ldots, a_{-1}$ where the orbit enters $\Slide[s]$ on negative time. The values $\sigma_i$ give the side they entered from.}
    \label{fig:3.8-3.2}
\end{figure}

\paragraph{\bfseries Step 4.} (Constructing the countable dense subset of orbits) For each $p \in E$, we will construct a countable set of orbits in $\wt M$ (see \cref{fig:3.8-4}). The union of theses orbits for all $p \in E$ will be our countable set. Let $p \in E$ and, for each $n \in \N$, define the subset $\Gamma_p^n \subseteq \wt M$ of all orbits $\gamma$ such that $\gamma(0)=p$ and, in the time interval $\left[-n, n\right]$, their $\Switch$-sequence satisfies $a_{k^-}, \ldots, a_{-1} \in E^\mathrm{ss}$ and $a_1, \ldots, a_{k^+} \in E^\mathrm{us}$. We define an equivalence relation in $\Gamma_p^n$ by determining that $2$ orbits are equivalent if their respective $\Switch$-sequences are equal in $\left[-n, n\right]$, and take $\wt E_p^n \subseteq \Gamma_p^n$ to be a set of orbits with one representative of each equivalence class. This set is countable, since in the finite interval $\left[-n, n\right]$ there is a maximum for all possible $k^+$, a minimum for all $k^-$, the $a_i$ belong to the countable set $E^\mathrm{ss} \cup E^\mathrm{us}$, and the $\sigma_i$ belong to the finite set $\{+1,-1\}$. We define
    \begin{equation*}
    \wt E := \bigcup_{p \in E} \bigcup_{n \in \N} \wt E_p^n.
    \end{equation*}
This is going to be our countable set dense on $\wt M$.

\begin{figure}[H]
    \centering
	\includegraphics{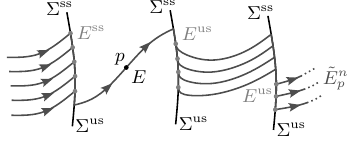}
    \caption{
    A countable set $\wt E_p^n$ of orbits is chosen, each orbit starting at a point $p \in E$ and, for a set time interval $\left[-n, n\right]$, entering $\Slide[s]$ through points of $E^\mathrm{ss}$ on negative time and leaving $\Slide[u]$ through points of $E^\mathrm{us}$ on positive time. The set is chosen such that all possible sequences of points of $E^\mathrm{ss}$ and $E^\mathrm{us}$ are taken into account.}
    \label{fig:3.8-4}
\end{figure}

\paragraph{\bfseries Step 5.} (Approximating orbits) Let $\gamma \in \wt M$ and $\varepsilon > 0$. We must find an orbit $\alpha \in \wt E$ such that $\wt d(\gamma, \alpha) < \varepsilon$ (see \cref{fig:3.8-5}). For this $\varepsilon$, we take $\delta > 0$ and $\tau > 0$ as in \cref{lemma:close.orbit.points.imply.close.orbits}. Let $a_0 := \gamma(0)$ and
    \begin{equation*}
    (a_{-k^-},\sigma_{-k^-}), \ldots, (a_{-1},\sigma_{-1}), (a_1,\sigma_1), \ldots, (a_{k^+},\sigma_{k^+}) 
    \end{equation*}
be the $\Switch$-sequence of $\gamma$ in the interval $\left[-\tau, \tau \right]$. Define $n = \lfloor \tau \rfloor$. For each $i \in \{-k^-,\ldots, k^+\}$, we will choose points $p_i \in M$ and numbers $\delta_i > 0$ as follows. If $a_0 \in \Slide \cup \Tange$, we choose $p_0 \in E^\mathrm{ss} \cup E^\mathrm{us}$; if $a_0 \in M \setminus (\Slide \cup \Tange)$ and the first time forwards that $\gamma$ reaches $\Slide[s] \cup \Tange$ is at a tangency point, or if the first time backwards that it reaches $\Slide[u] \cup \Tange$ is at a tangency point, we take $p_0 \in E^\mathrm{t}_p$, otherwise we take $p_0 \in E$. In all cases we choose $p_0$ really close to $a_0$, such that $d(a_0, p_0) < \delta_0$. Now, for each $i \in \{1, \ldots, k^+\}$, we choose $p_i \in E^\mathrm{us}$ really close to $a_i$, such that $d(a_i, p_i) < \delta_i$, considering that, if the first time $\gamma$ reaches $\Slide[s] \cup \Tange$ after leaving through $a_i$ is at a tangency point, we must take $p_i = a_i$. The same procedure must be carried on for the negative part of the orbit. Since $p_0 \in E$, $p_i \in E^\mathrm{us}$ for positive $i$ and $p_i \in E^\mathrm{ss}$ for negative $i$, there is a representative $\alpha \in \wt E_p^n$. Choosing all $\delta_i$ small enough, we can guarantee that, for all $t \in \left[-\tau, \tau\right]$, we have $d(\gamma(t), \alpha(t)) < \delta$, so it follows by \cref{lemma:close.orbit.points.imply.close.orbits} that $\wt d(\gamma, \alpha) < \varepsilon$.

\begin{figure}[H]
    \centering
	\includegraphics{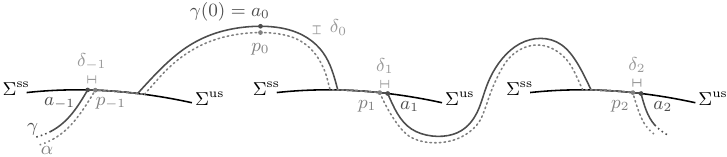}
    \caption{
    For a given orbit $\gamma$, we chose an orbit $\alpha \in \wt E$ that is close to it each time it enters $\Slide[s]$ on negative time, and leaves $\Slide[u]$ on positive time.
    }
    \label{fig:3.8-5}
\end{figure}

\end{proof}

We now proceed to prove that the orbit space $\wt M$ is a Baire space (\cref{prop:orbit.space.is.Baire}). We will need two lemmas (\cref{lem:uniform_tangencies,lem:complete_on_local_ball}).
Notice that, as a consequence of \cref{lemma:close.orbits.imply.close.orbit.points}, if $(\xi_n)_{n \in \N}$ is a Cauchy sequence in the orbit space $\wt M$ then, for each $t \in \R$, the sequence $(\xi_n(t))_{n \in \N}$ is a Cauchy sequence in the phase space $M$, so, since $M$ is complete, there exists a limit point $\xi_\infty(t) \in M$. Nonetheless, the function $\xi_\infty\colon \R \to M$ may fail to be a Filippov solution of $F$.

To see this, consider the example of the bean model (\cite{BCE}, \cref{fig:bean_model}). We may take a sequence of solutions of the bean model system that spin around the central tangent point $p$ with an increasingly higher frequency by concatenating loops that are closer and closer to $p$. Since every point in the next orbit gets closer to $p$, this sequence converges to a constant function $\xi_\infty$ that is stationary at $p$, but this in fact is not a solution of $F$, which means $\wt M$ is not always complete (see \cref{fig:bean_model_not_complete}).

\begin{figure}[H]
	\centering
	\includegraphics{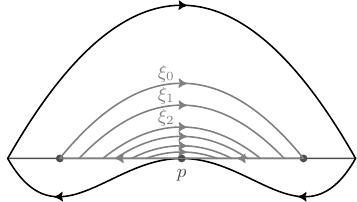}
	\caption[The orbit space of the bean model is not complete]{The orbit space of the bean model is not complete, since we can take a sequence of orbits $(\xi_n)_{n \in \N}$ that go around the central tangency point $p$ with increasing frequency. This sequence then converges to a stationary orbit at the tangency point $p$, which is not a Filippov orbit.}
	\label{fig:bean_model_not_complete}
\end{figure}

The problem in this example is that the distance between two times at which an orbit point intersects $\Tange$ converges to $0$. This is essentially what prevents completeness of $\wt M$. \Cref{lem:uniform_tangencies} serves the purpose of showing that we can control the distance between tangencies if we restrict the time interval of our solutions to a compact interval $[-n,n]$. \Cref{lem:complete_on_local_ball} shows that, in this domain $[-n,n]$, the restricted function $\xi_\infty|_{[-n,n]}\colon [-n,n] \to M$ (which is the pointwise limit of the Cauchy sequence $(\xi_n|_{[-n,n]})_{n \in \N}$) is a Filippov solution of $F$. This is a kind of restricted completeness which we will use to prove that $\wt M$ is Baire.

\begin{lemma}
\label{lem:uniform_tangencies}
Let $M$ be a manifold (possibly with boundaries) and $F$ be a bounded Filippov system on $M$ with $\Tange$ finite.
For every $\gamma \in \wt M$ and $n \in \N$, there exist positive real numbers $r$ and $\alpha$ such that, given $\xi \in \wtbola{\gamma}{r}$ and distinct $t, t' \in [-n,n]$, if $\xi(t), \xi(t') \in \Tange$, then $\abs{t' - t} > \alpha$. 
This is also true for every $0 < r' \leq r$.
\end{lemma}
\begin{proof}
In the compact interval $[-n,n]$, the orbit $\gamma$ goes through $\Tange$ a finite number of times, because $\Tange$ is finite. Let
	\begin{equation*}
	\beta := \inf\set{\abs{t' - t} \st t,t' \in [-n,n], t \neq t' \text{ and } \gamma(t), \gamma(t') \in \Tange}.
	\end{equation*}

From \cref{lemma:close.orbits.imply.close.orbit.points} (taking $\tau=n$ and $\delta$ small enough) it then follows that there exists $r>0$ such that, for every $\xi \in \wtbola{\gamma}{r}$, the point $\xi(t)$ is sufficiently close to $\gamma(t)$ for $t \in [-n,n]$, so by continuity of $F$ in the regular parts and because $\Tange$ is finite, we can guarantee that there is an $\alpha>0$ sufficiently close to $\beta$ such that, given $\xi \in \wtbola{\gamma}{r}$ and distinct $t, t' \in [-n,n]$, if $\xi(t), \xi(t') \in \Tange$, then $\abs{t' - t} > \alpha$.
This is also true for any positive $r' \leq r$ since $\wtbola{\gamma}{r'} \subseteq \wtbola{\gamma}{r}$.
\end{proof}

\begin{lemma}
\label{lem:complete_on_local_ball}
Let $M$ be a manifold (possibly with boundaries) and $F$ be a bounded Filippov system on $M$ with $\Tange$ finite.
For every $\gamma \in \wt M$ and $n \in \N$, let $r > 0$ be as in 
\cref{lem:uniform_tangencies}.
Let $(\xi_i)_{i \in \N}$ be a Cauchy sequence in $\wtbola{\gamma}{r}$ and $\xi_\infty(t)$ the (pointwise) limit of $(\xi_i(t))_{i \in \N}$ for $t \in [-n,n]$.
Then $\xi_\infty|_{[-n,n]}$ is a Filippov solution of $F$.
\end{lemma}
\begin{proof}
Let $\alpha$ be as in \cref{lem:uniform_tangencies} for $\gamma \in \wt M$ and $n \in \N$. We must show that the pointwise limit $\xi_\infty|_{[-n,n]}$ of $(\xi_i|_{[-n,n]})_{i \in \N}$ is a Filippov orbit.

First, take $t_0 \in [-n,n]$. If $\xi_\infty(t_0) \in M \setminus (\Slide \cup \Tange)$, the behavior of $\xi_\infty$ around $\xi_\infty(t_0)$ is the same as that of points in a regular system, so in this case $\xi_\infty$ is an orbit of the system in the neighborhood of $t_0$. We must study the behavior of $\xi_\infty$ when $\xi_\infty(t_0)$ belongs to $\Slide[s]$, $\Slide[u]$ or $\Tange$.

Let us first consider $\xi_\infty(t_0) \in \Slide[s]$. Take a tubular neighborhood $V \subseteq M$ of $\xi_\infty(t_0)$ small enough such that every point of $V$ flows to $\Slide[s]$ (this neighborhood exists because $\Slide[s]$ is open in $\Switch$). Since $\xi_i(t_0) \to \xi_\infty(t_0)$ as $i \to \infty$, there is a natural number $i_0 \in \N$ such that, for all $i \geq i_0$, $\xi_i(t_0) \in V$. For each of these $i$, take $\tau_i$ to be the smallest time $t \geq 0$ such that $\xi_i(t_0 + t) \in \Slide[s]$. Since $\xi_i(t_0) \to \xi_\infty(t_0) \in \Slide[s]$, we must have $\tau_i \to 0$ as $i \to \infty$. Now, since $\xi_i(t_0 + \tau_i)$ is on $\Slide[s]$, its flow is given by the slide flow $\Phi_{\mathrm s}$, so there is an open interval neighborhood $I = \left]0, d\right[$ such that, for all $t \in \left[0, d\right[$,
	\begin{equation*}
	\xi_i(t_0 + \tau_i + t) = \Phi_{\mathrm s}^t(\xi_i(t_0 + \tau_i)).
	\end{equation*}
Taking the limit as $i \to \infty$, we conclude that, for all $t \in \left[0, d\right[$,
	\begin{equation*}
	\xi_\infty(t_0+t) = \Phi_{\mathrm s}^t(\xi_\infty(t_0)).
	\end{equation*}
For negative $t$, there are $2$ cases. (1) If there is an open interval $I = \left]a, 0\right[$ of $0$ such that, for all $t \in I$, $\xi_\infty(t_0+t) \in \Slide[s]$; now take some $b \in \left]a, 0\right[$ and apply the previous proof for $\xi_\infty (t_0+b)$, which implies that $\xi_\infty(t_0 + t) = \Phi_{\mathrm s}^t(\xi_\infty(t_0))$ for all $t \in \left[b, 0\right[$. (2) In the other case, there is an open interval $I = \left]a, 0\right[$ such that, for all $t \in \left]a, 0\right[$, the limit points $\xi_\infty(t_0+t)$ belong to the regular region $M \setminus \Switch$ of the system, hence
	\begin{equation*}
	\xi_\infty(t_0+t) = \Phi_{\pm}^t(\xi_\infty(t_0)),
	\end{equation*}
where the $\pm$ sign indicates whether the flow is in $M_+$ or $M_-$. This proves that, in both cases, $\xi_\infty$ is an orbit around $\xi_\infty(t_0)$. The behavior around $\xi_\infty(t_0) \in \Slide[u]$ is the same as the previous case in $\Slide[s]$, just with the direction of the orbits inverted.

Finally, consider $\xi_\infty(t_0) \in \Tange$. Notice that $\Tange$ is finite, so the tangency points are isolated. Besides that, the $\alpha$ from \cref{lem:uniform_tangencies} guarantees that tangencies are all distanced at least $\alpha$ apart in time. So if we consider $t \in [t_0 - {\alpha}/{4}, t_0 + {\alpha}/{4}]$, it follows from \cref{lem:uniform_tangencies} that, for every $n \in \N$, each curve  $\xi_n|_{[t_0 - {\alpha}/{4}, t_0 + {\alpha}/{4}]}$ has at most $1$ tangency point. Since $\xi_\infty$ is an orbit around all points other than the tangency points, it follows that in a small enough neighborhood of $\xi_\infty(t_0)$ all points of $\xi_\infty$ are orbit points, so by continuity $\xi_\infty(t_0)$ is also an orbit point.
\end{proof}

\begin{proposition}
\label{prop:orbit.space.is.Baire}
Let $M$ be a manifold (possibly with boundaries) and $F$ be a bounded Filippov system on $M$ with $\Tange$ finite.
Then the orbit space $\wt M$ is a Baire space.
\end{proposition}
\begin{proof}
Let $(A_k)_{k \in \N}$ be a family of open dense sets in $\wt M$ and denote $D := \bigcap_{k \in \N} A_k$. We will show that $D$ is dense. Let $U \subseteq \wt M$ be an open set. We must show that $D \cap U \neq \emptyset$. Since $U$ is open and $A_0$ is open and dense, there exist an orbit $\gamma_0 \in U$ and a real number $0 < r_0$ such that $\wtbola{\gamma_0}{r_0} \subseteq U \cap A_0$.

We now construct inductively a sequence of orbits $(\gamma_n)_{n \in \N}$ and strictly positive real numbers $(r_n)_{n \in \N}$ such that
$\wt d(\gamma_n, \gamma_{n+1}) < \frac{r_n}{3}$, $\gamma_n \in A_n$, $\wtbola{\gamma_n}{r_n} \subseteq A_n$, and $r_n$ is a valid radius for the orbit $\gamma_n$ and integer $n$ in \cref{lem:complete_on_local_ball}. Let $n \in \N$ and assume the constructions works for integers less then or equal to $n$. Since $\wtbola{\gamma_n}{\frac{r_n}{3}}$ is open and $A_{n+1}$ is open and dense, there exists $\gamma_{n+1} \in \wtbola{\gamma_n}{\frac{r_n}{3}} \cap A_{n+1}$. Now choose a real number $r_{n+1} > 0$ small enough such that $r_{n+1} \leq \frac{r_n}{3}$, $\wtbola{\gamma_{n+1}}{r_{n+1}} \subseteq A_{n+1}$, and \cref{lem:complete_on_local_ball} holds for $\gamma_{n+1}$ and $n+1$.

The sequence of radii $(r_n)_{n \in \N}$ has the following properties. Induction on the relation $r_{n+1} \leq \frac{r_n}{3}$ implies that $r_{n+k} \leq \frac{r_n}{3^k}$. This shows that $(r_n)_{n \in \N}$ is strictly decreasing and $\lim_{n \to \infty} r_n = 0$.

We first show that the sequence $(\gamma_n)_{n \in \N}$ is a Cauchy sequence. To see this, notice that, for every $n \in \N$, we have $\wt d(\gamma_n, \gamma_{n+1}) < \frac{r_n}{3}$, so, for every $m \in \N$,
	\begin{equation*}
	\wt d(\gamma_n, \gamma_{n+m}) \leq \sum_{k=0}^{m-1} \wt d(\gamma_{n+k}, \gamma_{n+k+1}) < \sum_{k=0}^{m-1} \frac{r_{n+k}}{3} \leq \sum_{k=0}^{m-1} \frac{r_n}{3^{k+1}} < \frac{r_n}{2} < r_n.
	\end{equation*}
This shows that, for every $n,m \in \N$, $\gamma_{n+m} \in \wtbola{\gamma_n}{r_n}$. Now, given $\varepsilon > 0$, choose $N \in \N$ such that $r_N < \varepsilon$ (which is possible because $\lim_{n \to \infty} r_n = 0$) and it follows that, for every $m \geq n > N$, $\wt d(\gamma_n, \gamma_m) < r_n < r_N < \varepsilon$.
This shows that $(\gamma_n)_{n \in \N}$ is in fact a Cauchy sequence, so there exists a pointwise limit $\gamma_\infty$ of the sequence $(\gamma_n)_{n \in \N}$.

Now we show the pointwise limit $\gamma_\infty$ is a Filippov orbit. Fix $n \in \N$ and consider the Cauchy sequence $(\gamma_{n+m})_{m \in \N}$ in $\wtbola{\gamma_n}{r_n}$. By the choice of $r_n$, \cref{lem:complete_on_local_ball} implies that $\gamma_\infty|_{[-n,n]}$ is a Filippov solution of $F$. Since this is valid for every $n$, it follows that $\gamma_\infty$ is a Filippov solution of $F$, so $\gamma_\infty \in \wt M$.

We must finally prove that $\gamma_\infty \in D \cap U$. To this end we notice that, since $\gamma_\infty$ is the limit of $(\gamma_n)_{n \in \N}$, then we have
	\begin{equation*}
	\wt d(\gamma_n, \gamma_\infty) \leq \sum_{k=0}^{\infty} \wt d(\gamma_{n+k}, \gamma_{n+k+1}) < \sum_{k=0}^{\infty} \frac{r_{n+k}}{3} < \sum_{k=0}^{\infty} \frac{r_n}{3^{k+1}} = \frac{r_n}{2} < r_n.
\end{equation*}
so $\gamma_\infty \in \wtbola{\gamma_n}{r_n} \subseteq A_n$. This implies that $\gamma_\infty \in D = \bigcap_{n \in \N} A_n$, and also that $\gamma_\infty \in U$, since $\wtbola{\gamma_0}{r_0} \subseteq U \cap A_0$.
\end{proof}

Finally, we prove that the metric space $(\wt M, \wt d)$ is perfect, that is, has no isolated points.

\begin{proposition}
\label{prop:orbit.space.isolated.points}
Let $M$ be a manifold and $F$ be a bounded Filippov system on $M$.
The orbit space $\wt M$ is a perfect space.
\end{proposition}
\begin{proof}
Given any orbit $\gamma \in \wt M$, we can apply the flow to it for every $t \in \R$ to obtain a family of orbits $\{\wt \Phi^t(\gamma)\}_{t \in \R}$. To measure the distance between the original orbit and one of its translations, let us fix $\tau \in \R$. Since the Filippov vector field is bounded by $\nor{F}$ and each orbit is differentiable by parts, then from the mean value inequality it follows that, for any interval $I \subseteq \R$ and any orbit $\alpha \in \wt M$, we have $|\alpha(t) - \alpha(t')| \leq \nor{F}|t-t'|$. Then
    \begin{equation*}
    \wt d(\gamma, \wt \Phi^\tau(\gamma)) = \sum_{i \in \Z} \frac{1}{2^{|i|}} \sup_{i \leq t < t+1} |\gamma(t) - \gamma(\tau+t)| \leq 3\nor{F}|\tau|.
    \end{equation*}
Therefore we conclude that, as $\tau$ approaches $0$, $\wt \Phi^\tau(\gamma)$ approaches $\gamma$, so $\gamma$ is not an isolated point of $\wt M$.
\end{proof}

\begin{proof}[Proof of {\cref{theo:topological_properties_of_orbit_space}}]
This is just \cref{prop:orbit.space.is.separable,prop:orbit.space.is.Baire,prop:orbit.space.isolated.points}.
\end{proof}

\subsection{Transitivity}
\label{ssec:transitivity}

We now investigate the transitivity property.  The first proposition is an immediate consequence of the topological properties proven in the last subsection.

\begin{proposition}[Transitivity equivalence]
\label{theo:TopTransitivy}
Let $M$ be a manifold and $F$ a bounded Filippov system on $M$ such that $\Tange$ is finite and each tangency point has finite multiplicity.
For the flow $\wt \Phi^t$ on $\wt M$, topological transitivity is equivalent to transitivity (having a dense orbit).
\end{proposition}
\begin{proof}
The orbit space $\wt M$ is a perfect, separable and Baire metric space (\cref{prop:orbit.space.is.separable,prop:orbit.space.is.Baire,prop:orbit.space.isolated.points}), so this follows from Birkhoff's transitivity theorem (\cref{theo:birkhoff.transitivity}).
\end{proof}

\begin{proposition}
\label{prop:point.transitivity.above.implies.point.transitity.below}
Let $M$ be a manifold and $F$ be a bounded Filippov vector field on $M$. If the orbit space flow $\wt \Phi$ is transitive, then $F$ is also transitive.
\end{proposition}
\begin{proof}
To show $F$ is transitive, we must find an orbit $\gamma \in \wt M$ that is dense in $M$. Let $p \in M$ and $\delta > 0$, and take $\alpha \in \wt M$ such that $\alpha(0)=p$. Since the flow $\wt \Phi$ is transitive, there is a dense orbit $\Gamma$ in the orbit space $\wt M$. Define $\gamma: = \Gamma(0)$. For the given $\delta$ and any $\tau > 1$, take $\varepsilon > 0$ as in \cref{lemma:close.orbits.imply.close.orbit.points}. Since $\Gamma$ is dense in $\wt M$, there is $t_\alpha > 0$ such that $\wt d(\alpha, \wt \Phi^{t_\alpha}(\gamma)) < \varepsilon$, which implies by the choice of $\varepsilon$ that $|\alpha(t) - \wt \Phi^{t_\alpha}(\gamma)(t)| < \delta$ for every $|t| \leq 1$. In particular,
    \begin{equation*}
    |p - \gamma(t_\alpha)| = |\alpha(0) - \wt \Phi^{t_\alpha}(\gamma)(0)| < \delta,
    \end{equation*}
so $\gamma$ is dense in $M$.
\end{proof}

\begin{proposition}
\label{prop:main_transitivity}
Let $M$ be a connected manifold and $F$ be a Filippov system over $M$.
Suppose the set of tangency points $\Tange$ of the Filippov system is finite and there is a subset $\mathcal T \subseteq \Tange$ that satisfies the following properties:
    \begin{properties}
        \item \label{def:tangency.connections} For every $T, T' \in \mathcal T$, there is a $F$-orbit segment $\theta\colon [-s,s] \to M$ connecting $T$ to $T'$. If $T = T'$ we can consider $s=0$;
        \item \label{def:positive.return.tangency} For every orbit $\gamma$, there are a tangency point $T_\gamma \in \mathcal T$ and time $s_\gamma > 0$ such that $\gamma(s_\gamma) = T_\gamma$;
        \item \label{def:negative.return.tangency} For every orbit $\gamma$, there are a tangency point $T_\gamma \in \mathcal T$ and time $s_\gamma < 0$ such that $\gamma(s_\gamma) = T_\gamma$.
    \end{properties}
Then the flow $\wt \Phi$ is topologically transitive on the orbit space $\wt M$.
\end{proposition}
\begin{proof}
Let $\alpha, \beta \in \wt M$ be two orbits and $\varepsilon > 0$ a real number. To prove that the flow $\wt \Phi$ is topologically transitive, we must find an orbit $\gamma \in \wt M$ such that $\wt d(\alpha,\gamma) < \varepsilon$ and, for some $\tau > 0$, $\wt d(\beta, \wt \Phi^\tau(\gamma)) < \varepsilon$. Since the flow $\wt \Phi$ is invertible, this is the equivalent to finding $\tau_\alpha, \tau_\beta > 0$ such that $\wt d(\alpha, \wt \Phi^{-\tau_\alpha}(\gamma)) < \varepsilon$ and $\wt d(\beta, \wt \Phi^{\tau_\beta}(\gamma)) < \varepsilon$.

To construct the orbit $\gamma$, we first cut orbit $\alpha$ at a tangency point $T_\alpha \in \mathcal T$ and orbit $\beta$ at a tangency point $T_\beta \in \mathcal T$, and take an intermediary orbit segment $\theta$ connecting $T_\alpha$ to $T_\beta$ (by \cref{def:tangency.connections}). Then, we glue the negative orbit segment of $\alpha$ to the beginning of $\theta$, and the end of $\theta$ to the positive orbit segment of $\beta$. The resulting curve $\gamma$ is an orbit of the system, since it is made up of orbit segments connected at tangency points.

Take $\delta > 0$ and $\tau > 0$ with respect to the chosen $\varepsilon$, such as in \cref{lemma:close.orbit.points.imply.close.orbits}. We can take the smallest time $t_\alpha$ such that $t_\alpha \geq \tau > 0$ and $\alpha(t_\alpha) = T_\alpha$, for some $T_\alpha \in \mathcal T$ (by \cref{def:positive.return.tangency}), and the smallest time $t_\beta$ such that $-t_\beta \leq -\tau < 0$ and $\beta(-t_\beta) = T_\beta$, for some $T_\beta \in \mathcal T$ (by \cref{def:negative.return.tangency}).

Let $\theta\colon [-s,s] \to M$ be an orbit segment connecting $T_\alpha$ to $T_\beta$ (by \cref{def:tangency.connections}) and define orbit $\gamma$ to be
    \begin{equation*}
    \gamma (t) :=
        \begin{cases}
        \alpha(t + s + t_\alpha),& -\infty < t < -s \\
        \theta(t),& -s < t < s  \\
        \beta(t - s - t_\beta),& s < t < \infty.
        \end{cases}
    \end{equation*}
Since $\tau \leq t_\alpha$, this means $\wt \Phi^{-t_\alpha-s}(\gamma)$ coincides with $\alpha$ on the interval $\left] -\infty, \tau \right]$, so
    \begin{equation*}
    d(\alpha(t), \wt \Phi^{-t_\alpha-s}(\gamma)(t)) = 0 < \delta
    \end{equation*}
for every $|t| \leq \tau$, and likewise
    \begin{equation*}
    d(\alpha(t), \wt \Phi^{t_\beta+s}(\gamma)(t)) = 0 < \delta
    \end{equation*}
for every $|t| \leq \tau$. By the choice of $\delta$ and $\tau$ from \cref{lemma:close.orbit.points.imply.close.orbits}, this implies that $\wt d(\alpha, \wt \Phi^{-t_\alpha-s}(\gamma)) < \varepsilon$ and $\wt d(\beta, \wt \Phi^{t_\beta+s}(\gamma)) < \varepsilon$.
\end{proof}

We are able to apply the above result to the bean and sphere models \cite{BCE, EJV}.

\begin{example}
The dynamics of the orbit space of the bean model system \cite{BCE} is transitive.
\end{example}
\begin{proof}
The system has only $1$ tangency point $p$ (see \cref{fig:bean_model}) and every orbit pass through it infinite times, forward and backwards, so the hypothesis is satisfied.
\end{proof}

\begin{example}
The dynamics of the orbit space the sphere model system \cite{EJV} is transitive.
\end{example}
\begin{proof}
The system has $4$ tangency points; we choose $2$ of them, say $p_+$ or $p_-$ (see \cref{fig:sphere_model}), to be set $\mathcal T$. There is a closed orbit passing through both of them, which implies there are segments connecting one of them to the other, so the hypothesis is satisfied.
\end{proof}

\begin{proof}[Proof of {\cref{theo:transitivity_and_topological_transitivity_in_orbit_space}}]
This is a consequence of \cref{prop:point.transitivity.above.implies.point.transitity.below,prop:main_transitivity}. In particular, notice we do not assume transitivity of the Filippov field in \cref{prop:main_transitivity}, but rather obtain it as a consequence: a PSVF that satisfies the properties of \cref{prop:main_transitivity} has a transitive orbit space, so by \cref{prop:point.transitivity.above.implies.point.transitity.below} we conclude that it is transitive.
\end{proof}

\section{Further directions}

\begin{question}
Is there an example of a PSVF which is transitive, but is not transitive in the orbit space?
\end{question}

This question is an interesting and important one for the theory. It impacts, in spirit, how we understand PSVFs. If the answer is ``yes'', that is, if there is a PSVF which is transitive but not transitive in the orbit space, then we might have to use the orbit space to justify a more natural definition of transitivity. If the answer is ``no'', then both notions of transitivity are the same and hence the use of escaping regions to create transitivity is not actually an ``artificial'' transitivity.

\section*{Acknowledgements}
P. M. was partially financed by the Coordenação de Aperfeiçoamento de Pessoal de Nível Superior
Brasil (CAPES) - grant 141401/2020-6.
R. V. was partially financed by CNPq and Fapesp grants 18/13481-
0, 17/06463-3 and 2024/15612-6.

\printbibliography

\end{document}